\documentclass[a4paper,french,openright,11pt]{smfartVScomptage}

\usepackage{smfenum,smfthm,amsmath,amssymb,mathrsfs,euscript,color}
\usepackage{stmaryrd,mathabx,upgreek}
\usepackage[latin1]{inputenc}
\input{xypic}

\numberwithin{equation}{section} 
\theoremstyle{plain}

\def\ZZ{\mathbf{Z}}

\def\A{{\rm A}}
\def\B{{\rm B}}
\def\C{{\rm C}}
\def\D{{\rm D}}
\def\E{{\rm E}}
\def\F{{\rm F}}
\def\G{{\rm G}}
\def\H{{\rm H}}

\def\J{{\rm J}}
\def\K{{\rm K}}
\def\L{{\rm L}}

\def\N{{\rm N}}
\def\P{{\rm P}}
\def\Q{{\rm Q}}
\def\R{{\rm R}}

\def\T{{\rm T}}

\def\X{{\rm X}}
\def\Y{{\rm Y}}

\def\Aa{\mathscr{A}}

\def\Cc{\EuScript{C}}

\def\Oo{\EuScript{O}}
\def\Pp{\EuScript{P}}

\def\Tt{\EuScript{T}}

\def\Ga{\Gamma}
\def\La{\Lambda}

\def\b{\beta}
\def\d{\delta}
\def\e{\varepsilon}
\def\ee{\upepsilon}

\def\k{\kappa}
\def\l{\lambda}

\def\s{\sigma}
\def\t{\theta}

\def\w{\varpi}

\def\ie{c'est-à-dire }

\def\>{\geqslant}
\def\<{\leqslant}

\def\Mat{{\rm M}}
\def\GL{{\rm GL}}
\def\Gal{{\rm Gal}}

\def\mult#1{{#1}^{\times}}

\def\qlb{{\overline{\mathbf{Q}}_\ell}}
\def\zlb{{\overline{\mathbf{Z}}_\ell}}
\def\flb{{\overline{\mathbf{F}}_{\ell}}}

\def\r{{\textbf{\textsf{r}}}}

\def\cc{c}
\def\kk{\mathfrak{k}}
\def\ww{w}
\def\d{k}
\def\gg{g}
\def\sp{{\rm sp}}

\def\xt{x}
\def\ct{\widetilde\chi}
\def\kt{\widetilde\k}
\def\lt{\widetilde\l}
\def\pt{\widetilde\pi}
\def\rt{{\widetilde\rho}}
\def\st{\widetilde\s}
\def\tL{\widetilde\Lambda}

\def\GB{\mathcal{G}}
\def\TT{\boldsymbol{\Theta}}
\def\0{\boldsymbol{0}}
\def\YY{\textbf{\textsf{Y}}}
\def\BB{\textbf{\textsf{B}}}

\long\def\MSC#1\EndMSC{\def\arg{#1}\ifx\arg\empty\relax\else
     {\par\narrower\noindent%
     2010 Mathematics Subject Classification: #1\par}\fi}

\long\def\KEY#1\EndKEY{\def\arg{#1}\ifx\arg\empty\relax\else
	{\par\narrower\noindent Keywords and Phrases: #1\par}\fi}

\title[Comptage de représentations cuspidales congruentes]
{Comptage de représentations cuspidales congruentes}
\author{Vincent Sécherre}
\address{Université de Versailles St-Quentin-en-Yvelines\\
Laboratoire de Mathémati\-ques de Versailles\\
45 avenue des Etats-Unis\\
78035 Versailles cedex, France}
\email{vincent.secherre@math.uvsq.fr}

\begin{altabstract}
Let $\F$ be a non-Archimedean locally compact field of residue characteristic 
$p$, $\G$ be an inner form of $\GL_n(\F)$, $n\>1$, and $\ell$ be a 
prime number different from $p$.
We give a numerical criterion for an integral $\ell$-adic irreducible cuspidal 
representation $\rt$ of $\G$ to have a super\-cuspidal irreducible reduction 
mod $\ell$, by counting inertial classes of cuspidal representations that are 
congruent to the inertial class of $\rt$, generalizing results by Vignéras and 
Dat. 

In the case the reduction mod $\ell$ of $\rt$ is not super\-cuspidal 
irreducible, we show that this counting argument allows us to compute
its length and the size of the supercuspidal support of its irreducible 
components. 
We define an invariant $w(\rt)\>1$ | the product of this length by this 
size | which is expected to behave nicely through the local Jacquet-Langlands 
correspondence. 

Given an $\ell$-modular irreducible cuspidal representation $\rho$ of $\G$ 
and a positive integer $a$, we give a criterion for the existence of an integral 
$\ell$-adic irreducible cuspidal representation $\rt$ of $\G$ such that its 
reduction mod $\ell$ contains $\rho$ and has length $a$.
This allows us to obtain a formula for the cardinality of the set of 
reductions mod $\ell$ of 
inertial classes of $\ell$-adic irreducible cuspidal representations $\rt$ 
with given depth and invariant $w$. 

These results are expected to be useful to prove that the local 
Jacquet-Langlands correspondence preserves congruences mod $\ell$.
\end{altabstract}

\thanks{Mon séjour à l'University of East Anglia en mai 2014, 
durant lequel une partie de ce travail a été effectuée, 
a été financé par l'EPSRC (grant EP/H00534X/1).
Je remercie chaleureusement Shaun Stevens pour son invitation 
et pour les discussions que nous avons eues à propos de ce travail.}

\begin{document} 

\maketitle

\MSC 
22E50 
\EndMSC

\KEY 
Modular representations of $p$-adic reductive groups, 
Jacquet-Lang\-lands correspondence, Cuspidal representations,
$\ell$-adic lifting, Congruences mod $\ell$
\EndKEY

\thispagestyle{empty}


\section{Introduction et énoncé des principaux résultats}

\subsection{} 

Soit $\F$ un corps localement compact non archimédien, 
soit $\H$ le groupe linéaire général $\GL_{n}(\F)$, $n\>1$, 
et soit $\A$ le groupe multiplicatif d'une 
$\F$-algèbre à division centrale de degré réduit égal à $n$.
Soit $\ell$ un nombre premier différent de la caractéristique résiduelle
$p$ de $\F$. 
Dat a construit dans \cite{Datj} un cas particulier de correspondance de 
Jacquet-Langlands locale modulo $\ell$.
C'est une bijection entre classes de représentations 
irréductibles $\ell$-modulaires --- \ie à coefficients dans une 
clôture algébrique $\flb$ d'un corps fini de caractéristique $\ell$ --- 
de $\A$ et classes de certaines représentations irréductibles 
$\ell$-modulaires de $\H$, baptisées ``super-Speh'' 
(paragraphe \ref{Finale}).
Elle est compatible, en un certain sens, à la correspondance de 
Jacquet-Lang\-lands locale entre les classes de re\-pré\-sen\-ta\-tions 
irréductibles $\ell$-adiques --- \ie à coefficients dans une 
clôture algé\-brique $\qlb$ du corps des nombres $\ell$-adiques --- 
de $\A$ et la série discrète $\ell$-adique de $\H$.

\subsection{} 

Plutôt que d'étudier directement la série discrète $\ell$-adique, 
qui se réduit mal modulo $\ell$, 
Dat étu\-die son image par l'involution de Zelevinski, \ie l'ensemble des
classes de représen\-ta\-tions de Speh $\ell$-adiques de $\H$.
Une telle représentation est dite 
$\ell$-super-Speh lorsqu'elle est entière, et lors\-que sa réduction modulo $\ell$ est 
irréductible et super-Speh.
La construction de la correspondan\-ce de Jacquet-Langlands locale 
modulo $\ell$ de \cite{Datj} repose sur le fait crucial que la correspondance 
de Jacquet-Langlands loca\-le $\ell$-adique fait se correspondre bijectivement 
représentations $\ell$-adiques entières de $\A$ dont la réduction modulo $\ell$ est 
irréductible, et représentations $\ell$-super-Speh de $\H$.

\subsection{} 

Pour prouver ce fait, Dat utilise un critère numérique de
$\ell$-supercuspidalité établi par Vignéras pour construire 
une correspondance de Langlands locale modulo $\ell$ (\cite{Vigl}).
La réduction modulo $\ell$ d'une représentation irréductible cuspidale 
$\ell$-adique entière $\rt$ de $\H$ est toujours irréductible et cuspidale, 
mais elle n'est pas toujours supercuspidale~; 
plus précisément, elle est supercuspidale si et seu\-le\-ment si 
le nombre de représentations 
cuspidales $\ell$-adi\-ques entières de $\H$ qui sont 
\textit{stric\-te\-ment} congrues à $\rt$ est 
``le plus grand possible'' (voir \cite[Proposition 2.3]{Vigl}).
Il prouve aussi une variante de ce critère numérique pour $\A$ 
(voir \cite[Proposition 2.3.2]{Datj}).

\subsection{} 
\label{P14}

Soit maintenant $\D$ une $\F$-algèbre à division centrale de degré réduit $d$, 
et posons $\G=\GL_{m}(\D)$, $m\>1$.
Si l'on veut construire une correspondance de Jacquet-Langlands locale modulo 
$\ell$ généra\-le, il est naturel de commencer par généraliser le critère
numérique de $\ell$-supercuspidalité de Vi\-gné\-ras et Dat. 
C'est ce que nous faisons, en le présentant sous une forme légèrement 
diffé\-rente.
Soit $\rt$ une représentation irréductible cuspidale $\ell$-adique entière de
$\G$. 
Il y a (\cite[Théorème 3.15]{MSt}) 
une re\-pré\-sentation irréductible cuspidale $\ell$-modulaire 
$\rho$ de $\G$ et un unique entier $a=a(\rt)\>1$ tels que la réduction modulo 
$\ell$ de $\rt$ soit égale à~: 
\begin{equation}
\label{redcusp}
\r_\ell(\rt)=\rho+\rho\nu+\dots+\rho\nu^{a-1}
\end{equation}
dans le groupe de Grothendieck des représentations $\ell$-modulaires de 
longueur finie de $\G$, où $\nu$ désigne le caractère valeur absolue de la 
norme réduite.
La repré\-sentation $\rho$ n'est pas unique en gé\-né\-ral, mais sa classe 
inertielle $[\G,\rho]$ ne dépend que de la clas\-se inertielle $[\G,\rt]$ de $\rt$.
Quand $\G$ est déployé, l'entier $a(\rt)$ est toujours égal à $1$, 
\ie que la ré\-duction modulo $\ell$ de $\rt$ est tou\-jours irréductible.

\begin{defi}
On dit que $\rt$ est \textit{$\ell$-supercuspidale} 
si $\r_\ell(\rt)$ est irréductible et supercuspidale.
\end{defi}

\subsection{} 

Etant donnée une représentation irréductible cuspidale 
$\rt$ comme en \ref{P14},
on note~:
\begin{equation*}
\r_\ell([\G,\rt]) 
\end{equation*}
l'ensemble des réductions modulo $\ell$ des représentations 
entières iner\-tiellement équivalentes à $\rt$, 
et on appelle cet ensemble la \textit{réduction modulo $\ell$} de $[\G,\rt]$. 
On note $n(\rt)$ le nom\-bre de ca\-rac\-tè\-res $\ell$-adiques 
non ramifiés $\ct$ de $\G$
tels que $\rt\ct$ est isomorphe à $\rt$ et $\cc(\rt)$
la plus grande puis\-san\-ce de $\ell$ divisant $q^{n(\rt)}-1$. 
Le résultat suivant généralise \cite{Vigl} et \cite{Datj}.

\begin{theo}
\label{CongCusp}
Soit $\rt$ une représentation irréductible cuspidale $\ell$-adique entière de $\G$. 
\begin{enumerate}
\item
L'ensemble des classes inertielles de représentations 
irréductibles cuspidales $\ell$-adi\-ques en\-tiè\-res de $\G$ congrues à $\rt$
est fini, de cardinal noté $t(\rt)$.
\item
On a~:
\begin{equation*}
t(\rt)\<\cc(\rt)
\end{equation*}
avec égalité si et seulement si $\rt$ est $\ell$-supercuspidale.
\end{enumerate}
\end{theo}

\subsection{} 
\label{DEFK}

Intéressons-nous maintenant au cas où $\rt$ n'est pas
$\ell$-super\-cuspidale~;
en étu\-diant plus finement la façon dont les entiers $t(\rt)$ et $c(\rt)$ diffèrent, 
il est raisonnable de penser qu'on pourra en déduire des in\-for\-mations
sur la structure de $\rt$.
D'après la classification des représentations irréductibles cuspidales 
$\ell$-modulaires de $\G$ en 
fonction des supercuspidales (\cite[Théorème 6.14]{MSc}), 
il existe un unique entier naturel~:
\begin{equation*}
\d(\rho) \>1
\end{equation*}
tel que $\rho$ apparaisse comme sous-quotient de l'induite parabolique d'une 
représentation irré\-duc\-tible supercuspidale du sous-groupe de Levi standard 
$\GL_{r}(\D)\times\dots\times\GL_{r}(\D)$ avec $r\d(\rho)=m$.
En particulier, $\rho$ est supercuspidale si et seulement si $\d(\rho)=1$.
Posons~:
\begin{equation}
\label{DEFw}
\ww(\rt) = \d(\rho) a(\rt).
\end{equation}
Ainsi $\rt$ est $\ell$-super\-cuspidale si et seulement si $\ww(\rt)=1$. 
Le résultat suivant montre qu'on peut déterminer la valeur de $\ww(\rt)$ 
en com\-pa\-rant $t(\rt)$ et $\cc(\rt)$.

\begin{theo}
\label{CalculW}
Soit $\rt$ une $\qlb$-représentation irréductible cuspidale entière et 
non $\ell$-super\-cuspidale de $\G$. 
Alors~:
\begin{equation*}
t(\rt) \ww(\rt) = 
\left\{
\begin{array}{ll}
\cc(\rt)-1 & \text{si $t(\rt)$ est premier à $\ell$,} \\
\cc(\rt)(\ell-1)\ell^{-1} & \text{sinon.}
\end{array}
\right.
\end{equation*}
\end{theo}

\subsection{} 
\label{par7}

Changeons maintenant de point de vue. 
Quand $\G$ est déployé, 
Vignéras a montré (\cite{Vigb}) qu'une représentation irréductible
cus\-pi\-dale $\ell$-modu\-lai\-re $\rho$ de $\G$ 
se relève toujours en une 
représentation $\ell$-adique de $\G$, 
\ie qu'il existe une re\-pré\-sentation $\ell$-adique entière de $\G$ dont la 
réduction modulo $\ell$ est isomorphe à $\rho$. 
Si main\-te\-nant $\G$ n'est pas déployé, toute représentation 
irréductible \textit{super\-cuspidale}
$\ell$-mo\-du\-laire de $\G$ se relève à $\qlb$
(voir \cite{MSt,MSc})
mais il existe 
des représentations cuspida\-les qui ne se relèvent pas.
Etant don\-née une représentation 
cuspidale non super\-cus\-pidale $\ell$-mo\-dulaire $\rho$ de $\G$,
il est naturel de demander à quelle condition elle admet un relè\-ve\-ment.

\subsection{} 
\label{DEFds}

Pour répondre à cette question, nous avons besoin de l'invariant~:
\begin{equation*}
s(\rho) \>1
\end{equation*} 
introduit dans \cite{MSt},
dont la définition repose sur la construction des
représentations irré\-duc\-ti\-bles cuspidales de $\G$ par la théorie des types 
de Bushnell-Kutzko (voir la section \ref{SECTION2} ci-dessous). 
C'est un diviseur de $d$~; en particulier il est toujours égal à $1$ quand 
$\G$ est déployé.
Cet invariant est relié à un autre invariant, le 
\textit{degré para\-métrique} $\delta(\rho)$ 
introduit dans \cite{BH}, par l'identité $\delta(\rho)s(\rho)=md$.

\begin{theo}
\label{lift}
Soit $\rho$ une représentation irréductible cuspidale non 
supercuspidale $\ell$-modu\-lai\-re de $\G$.
Pour que $\rho$ se relève à $\qlb$, il faut et il suffit que 
les entiers $s(\rho)$ et $\d(\rho)$ soient premiers entre eux
et que la représentation tordue $\rho\nu$ soit isomorphe à $\rho$.
\end{theo}

Quand $\G$ est déployé, on a 
toujours $s(\rho)=1$ et une repré\-sen\-tation irréductible cuspidale non 
super\-cuspidale $\rho$ est toujours isomorphe à sa tordue $\rho\nu$.
La condition du théo\-rème \ref{lift}
est donc toujours vérifiée~; on retrouve ainsi le résultat de Vignéras du 
paragraphe \ref{par7}. 

\subsection{} 

Plus généralement, une représentation irréductible cus\-pi\-dale 
$\ell$-modu\-lai\-re $\rho$ de $\G$ étant fixée, nous pouvons 
chercher les 
valeurs possibles de $w(\rt)$ lorsque $\rt$ décrit les représentations 
irré\-duc\-tibles cuspidales $\ell$-adiques entières 
de $\G$ dont la réduction modulo $\ell$ contient $\rho$. 
Le théorème \ref{lifta}
répond à cette question et complète ainsi le théorème \ref{lift}.
Notons $v$ la valuation $\ell$-adique sur $\ZZ$
(normalisée par $v(\ell)=1$) et notons $\e(\rho)$ l'ordre de $q^{n(\rho)}$ dans 
$(\ZZ/\ell\ZZ)^\times$, \ie le plus 
petit entier $k\>1$ tel que $\rho\nu^k$ soit isomorphe à $\rho$.

\begin{theo}
\label{lifta}
Soit $\rho$ une représentation irréductible cuspidale $\ell$-modu\-lai\-re de 
$\G$ et soit un entier $a>1$.
Pour qu'il existe une $\qlb$-représentation irréductible cuspidale entière 
$\rt$ de $\G$ dont la réduc\-tion modulo $\ell$ contienne $\rho$ et soit de 
longueur $a$, il faut et il suffit que~:
\begin{enumerate}
\item
il existe un entier $u\in\{0,\dots,v(s(\rho))\}$ tel que $a=\e(\rho)\ell^u$~;
\item
les entiers $s(\rho)a^{-1}$ et $\d(\rho)$ soient premiers entre eux. 
\end{enumerate}
\end{theo}

\subsection{} 

Dans la dernière section, nous utilisons le théorème \ref{lifta} pour 
obtenir une formule de comptage de classes inertielles de représentations 
cuspidales $\ell$-modulaires, dans l'esprit de~\cite{BHcgds}. 
Contrairement à Bushnell et Henniart, qui obtiennent leur formule en 
s'appuyant sur la correspondance de Jacquet-Langlands locale et sur 
l'existence préalable 
d'une telle for\-mu\-le dans le cas 
du groupe mul\-tiplicatif d'une algèbre à division,
nous établissons la nôtre par un calcul direct, en termes 
de $\F$-endo\-classes de caractères simples \cite{BSS}.

Fixons des entiers $n,w\>1$ tels que $w$ divise $n$, et un nombre rationnel 
$j\>0$. 
Soient $m,d\>1$ des entiers tels que $md=n$
et soit $\D$ une $\F$-algèbre à division centrale de degré réduit $d$.
On note $\Aa_\ell(\D,j,w)$ l'en\-sem\-ble des réductions mod $\ell$ de 
classes inertielles de re\-pré\-sen\-ta\-tions ir\-ré\-ductibles cuspidales 
$\ell$-adiques $\rt$ telles que~: 
\begin{enumerate}
\item
il existe un entier $u\>1$ divisant $m$ 
tel que $\rt$ soit une représentation irréductible
cuspidale $\ell$-adique de $\GL_{u}(\D)$~;
\item
on a $w(\rt)=w$ et le niveau normalisé de $\rt$ est inférieur ou égal à $j$.
\end{enumerate}
C'est un ensemble fini, de cardinal noté 
$\boldsymbol{a}_{\ell}(\D,j,w)$.
Fixons par ailleurs une clôture al\-gé\-brique $\overline{\kk}$ du corps résiduel
$\kk_\F$ de $\F$, et notons 
$\boldsymbol{y}_{\ell}^{1}(q,n,w)$ le nom\-bre de 
$y\in\overline{\kk}{}^\times$ tels que~:
\begin{enumerate}
\item
l'ordre de $y$ est premier à $\ell$~;
\item
le degré de $y$ sur $\kk_\F$, noté $\deg(y)$, divise $nw^{-1}$~; 
\item
l'ordre de $q^{\deg(y)}$ 
dans $(\ZZ/\ell\ZZ)^\times$ est égal au 
plus grand diviseur de $w$ premier à $\ell$. 
\end{enumerate}
On a la formule suivante~;
pour la notion d'endo-classe, on renvoie au paragraphe \ref{ENDO} ci-dessous et à 
\cite{BSS}.

\begin{theo}
\label{comptagedeclasses}
On a~:
\begin{equation*}
\boldsymbol{a}_{\ell}(\D,j,w) 
= \sum\limits_{\TT}\ \boldsymbol{y}_{\ell}^{1}(q(\TT),n(\TT),w),
\end{equation*}
la somme portant sur les $\F$-endoclasses $\TT$ de niveau normalisé inférieur 
ou égal à $j$ et de degré $\deg(\TT)$ divisant $nw^{-1}$, et où~:
\begin{equation*}
n(\TT)=\frac{n}{\deg(\TT)},
\quad
q(\TT)=q^{f(\TT)},
\end{equation*}
l'entier $f(\TT)$ désignant le degré résiduel de $\TT$.
\end{theo}

Cette somme ne dépendant que de $\ell$, $n$, $w$, $j$ et $q$, 
on en déduit le corollaire suivant.

\begin{coro}
\label{DonaTartt}
On a $\boldsymbol{a}_{\ell}(\D,j,w)=\boldsymbol{a}_{\ell}(\F,j,w)$. 
\end{coro}

\subsection{}
\label{Finale}

Dans un travail ultérieur, nous montrerons comment le critère numérique 
affiné (théorème \ref{CalculW}) et la formule de comptage (corollaire 
\ref{DonaTartt}) jouent un rôle central pour prouver que la correspon\-dan\-ce 
de Jacquet-Langlands locale $\ell$-adique entre représentations irréductibles 
entières de $\A$ et sé\-rie dis\-crè\-te entière de $\G$ préserve les congruences 
et l'invariant $w$, induisant une bi\-jec\-tion entre clas\-ses de 
congruence modulo $\ell$ de représentations. 
Les résultats de Dat \cite{Datj} correspondent au cas particulier 
où $\G=\H=\GL_n(\F)$ et $w=1$.



\section*{Notations et conventions}

Dans tout cet article, on fixe un corps localement compact non archimédien 
$\F$ et une $\F$-algèbre à division centrale $\D$, de degré réduit noté $d$. 
On fixe un entier $m\>1$ et on pose $\G=\GL_m(\D)$.

Si $\K$ désigne une algèbre à division sur une extension finie de $\F$, 
on note $\kk_\K$ son corps résiduel et $q_\K$ le cardinal de $\kk_\K$.
On note en particulier $q=q_\F$ le car\-di\-nal du corps résiduel de $\F$.

On fixe un nombre premier $\ell$ ne divisant pas $q$.
On note $\qlb$ une clôture algébrique du corps des nombres $\ell$-adiques
et $\zlb$ son anneau d'entiers.
Son corps résiduel $\flb$ 
est une clôture algé\-brique d'un corps fini de caractéristique $\ell$.

Toutes les représentations considérées dans cette article sont lisses. 

Deux représentations $\ell$-adiques entières de longueur finie 
(de $\G$ ou d'un groupe profini) 
sont dites congruentes (modulo $\ell$) 
si elles ont la même réduction modulo $\ell$
(voir \cite{Vigb,Vigw}). 

\section{Rappels sur les représentations cuspidales}
\label{SECTION2}

Au paragraphe \ref{PARA21}, $\R$ est un corps algébriquement clos de 
caractéristique $0$ ou $\ell$.  
Ensuite, ce sera ou bien le corps $\qlb$, ou bien le corps $\flb$. 

\subsection{}
\label{PARA21}

Rappelons quelques faits tirés de \cite{MSt} sur les 
$\R$-représentations irréductibles cuspidales de $\G$.
D'abord, il y a une correspondance bijective naturelle~:
\begin{equation}
\label{sichuan}
[\G,\rho] \leftrightarrow [\J,\l]
\end{equation}
entre classes 
inertielles de $\R$-représentation irréductible cuspidale de $\G$ et 
classes de $\G$-conjugai\-son d'objets appelés types simples maximaux 
de $\G$ (\cite[\S3]{MSt}).
Plus précisément, la classe d'inertie de $\rho$ et la classe de conjugaison 
de $(\J,\l)$ se correspondent par \eqref{sichuan} 
si et seulement si la restric\-tion de $\rho$ à
$\J$ possède une sous-représentation isomorphe à $\l$. 

Un type simple maximal de $\G$ est une paire $(\J,\l)$ formée d'un 
sous-groupe ouvert compact $\J$ de $\G$ et d'une $\R$-repré\-sen\-ta\-tion 
irréductible $\l$ de $\J$ dont la construction est effectuée en \cite[\S2]{MSt}.
Résumons-en brièvement les prin\-ci\-pa\-les étapes. 

D'abord, on part d'une strate simple $[\La,n,0,\b]$ dans la $\F$-algèbre de
matrices $\Mat_m(\D)$ 
et d'un caractère sim\-ple $\t\in\Cc(\La,0,\b)$ d'un sous-groupe ouvert 
compact $\H^1=\H^1(\b,\La)$ de $\G$.
Il y a un sous-groupe ouvert compact $\J^1=\J^1(\b,\La)$ de $\G$ contenant 
et normalisant $\H^1$, et possédant une unique re\-pré\-sentation irréductible $\eta$ 
dont la restriction à $\H^1$ contienne $\t$.

La représentation $\eta$ se prolonge en une représentation irréductible $\k$ d'un 
sous-groupe ouvert compact $\J=\J(\b,\La)$ de $\G$ contenant 
et normalisant $\J^1$, de même ensemble d'entrelacement que $\eta$~;
un tel prolongement $\k$ s'appelle une \textit{$\b$-extension} de $\eta$.

On suppose que $\J\cap\mult\B$ est un sous-groupe compact
maximal de $\mult\B$.
On fixe un isomorphisme d'algèbres entre le centralisateur $\B$ de $\E=\F[\b]$ 
dans $\Mat_m(\D)$ et une $\E$-algèbre $\Mat_{m'}(\D')$ pour un $m'\>1$ 
et une $\E$-algèbre 
à division centrale $\D'$ convenables, iden\-ti\-fiant $\J\cap\mult\B$ au
sous-groupe compact maximal standard de $\GL_{m'}(\D')$.

Le groupe $\J$ est égal à $(\J\cap\mult\B)\J^1$, et on a des isomorphismes de 
groupes~: 
\begin{equation*}
\label{iso}
\J/\J^1 \simeq (\J\cap\mult\B)/(\J^1\cap\mult\B) \simeq \GL_{m'}(\kk_{\D'}),
\end{equation*}
le second étant induit par l'isomorphisme de $\E$-algèbres fixé précédemment. 
Notons $\GB$ ce dernier groupe et fixons une re\-présentation irréductible 
cuspidale $\s$ de $\GB$.
Elle définit, par inflation, une re\-pré\-sen\-tation irréductible de $\J$ triviale sur $\J^1$,
encore notée $\s$. 
Alors la paire $(\J,\k\otimes\s)$ est un type simple maximal de $\G$, et tous 
sont construits de cette façon.

Soit $\rho$ une $\R$-représentation irréductible cuspidale de $\G$ dont la 
classe inertielle $[\G,\rho]$
corres\-pond à la classe de conjugaison d'un type simple maximal $(\J,\l)$. 
Le groupe de Galois de $\kk_{\D'}$ sur $\kk_{\E}$ agit sur les représentations
de $\GB$~; on note~:
\begin{equation*}
s(\rho) = s(\s)
\end{equation*}
l'ordre du stabilisateur de $\s$ dans ce groupe de Galois.
Quand $\G$ est deployé, \ie quand $\D$ est égale à $\F$, 
ce groupe de Galois est trivial et on a toujours $s(\rho)=1$.

Notons $n(\rho)$ 
le nombre de caractères non ramifiés $\chi$ de $\G$ tels que la 
représentation tordue $\rho\chi$ soit isomorphe à $\rho$, 
et $f(\rho)$ le quotient de $md$ par l'indice de ramification 
de $\E$ sur $\F$.

Ces trois entiers sont indépendants des choix 
effectués dans la construction de $(\J,\l)$~; ils ne dépendent 
que de la classe inertielle de $\rho$.
Lorsque $\R$ est de caractéristique $0$, ils sont liés par la relation~: 
\begin{equation*}
n(\rho)s(\rho) = f(\rho).
\end{equation*}
Lorsque $\R$ est de caractéristique $\ell$ en revanche, 
l'entier $s(\rho)$ divise $f(\rho)$, 
et $n(\rho)$ est le plus grand diviseur de $f(\rho)s(\rho)^{-1}$ 
premier à $\ell$.

Lorsque $\R$ est de caractéristique $\ell$,
on a introduit au paragraphe \ref{DEFK} un entier $\d(\rho)$~: 
le nombre de termes dans le support supercuspidal de $\rho$. 
De façon analogue, 
il y a un unique entier naturel~:
\begin{equation*}
\d(\s)
\end{equation*}
tel que $\s$ apparaisse comme sous-quotient de l'induite parabolique d'une 
représentation irré\-duc\-tible supercuspidale du sous-groupe de Levi standard 
$\GL_{r}(\kk_{\D'})\times\dots\times\GL_{r}(\kk_{\D'})$ avec $r\d(\s)=m'$.

\begin{lemm}
On a $\d(\rho)=\d(\s)$.
\end{lemm}

\begin{proof}
Posons $\d=\d(\rho)$ et définissons un entier $r\>1$ par $\d r=m$.
Il y a une repré\-sen\-ta\-tion irréductible supercuspidale $\rho_0$ et des 
$\R$-caractères non ramifiés
$\chi_1,\dots,\chi_{\d}$ de $\GL_{r}(\D)$ tels que $\rho$ apparaisse comme un 
sous-quotient de l'induite parabolique de 
$\rho_0\chi_1\otimes\dots\otimes\rho_0\chi_k$ à $\G$
(\cite[Théorème 6.14]{MSc}).
Fixons un type simple maximal $(\J_0,\k_0\otimes\s_0)$ contenu dans $\rho_0$. 
D'après, par exemple, la preuve de \cite[Lemme 6.1]{MSc}, 
on peut choisir $\s_0$ de sorte que $\s$ apparaisse comme un sous-quotient de 
l'induite parabolique de 
$\s_0\otimes\dots\otimes\s_0$ à $\GB$.
D'après \cite[Proposition 6.10]{MSc}, la re\-pré\-sentation $\s_0$ est 
supercuspidale.
Par unicité de $\d(\s)$, on en déduit que $\d(\s)=\d$.
\end{proof}

\begin{rema}
\label{gavroche}
On en déduit que $\d(\rho)$ divise $m'$, et pas seulement $m$.
\end{rema}

\subsection{}
\label{PARA22}

Fixons une extension $\kk$ de $\kk_{\D'}$ de degré $m'$, 
et notons $\X$ l'ensemble des $\xt\in\mult\kk$ de degré $m'$ sur 
$\kk_{\D'}$.  
D'après Green \cite{Green}, il y a une application surjective~:
\begin{equation}
\label{green}
\xt\mapsto\boldsymbol{\st}(\xt)
\end{equation}
de $\X$ vers l'ensemble des classes de représen\-ta\-tions irréduc\-ti\-bles 
cuspidales $\ell$-adiques de $\GB$~;
les antécédents de $\boldsymbol{\st}(x)$ sont les conjugués de $x$
sous $\Gal(\kk/\kk_{\D'})$.
Pour $x\in\mult\kk$, notons $[x]$ l'orbite de $x$ sous $\Gal(\kk/\kk_\E)$ et~: 
\begin{equation*}
\deg(x) = {\rm card}([\xt])
\end{equation*}
le degré de $x$ sur $\kk_\E$.
Notons $d'$ le degré réduit de $\D'$ sur $\E$.

\begin{lemm}
\label{lemsm}
Pour $x\in\X$, soit $\st$ la représentation cuspidale lui
correspondant par \eqref{green}.
On a la relation~: 
\begin{equation}
\label{DEGX}
\deg(x) 
= \frac{m'd'}{s(\st)}
\end{equation}
et $s(\st)$ est premier à $m'$.
\end{lemm}

\begin{proof}
Notons $\phi$ l'automorphisme de Frobenius $x\mapsto x^{q_\E}$.
Pour $k\in\ZZ$, on a~:
\begin{equation*}
\st^{\phi^k}\simeq\st
\quad\Leftrightarrow\quad
\text{il existe $l\in\ZZ$ tel que $x^{q_\E^k}=x^{q_{\D'}^l}$}.
\end{equation*}
Si l'on note $[\st]$ l'orbite de $\st$ sous $\Gal(\kk_{\D'}/\kk_\E)$, 
on en déduit que~:
\begin{equation*}
{\rm card}([\st]) 
= \frac{d'}{s(\st)}
= (d',\deg(x)).
\end{equation*}
Par ailleurs, si $n$ est l'ordre de $x$,
alors l'ordre de $q_\E$ dans $(\ZZ/n\ZZ)^\times$ est $\deg(x)$, 
tandis que l'ordre de $q_{\D'}$ dans $(\ZZ/n\ZZ)^\times$ est $m'$.
On en déduit le résultat voulu. 
\end{proof}

\begin{coro}
\label{eponine}
L'entier $s(\st)$ est premier à $m$.
\end{coro}

\begin{proof}
Notons $g$ le degré de $\E$ sur $\F$, de sorte que~:
\begin{equation*}
d' = \frac{d}{(d,\gg)},
\quad
m = m' \cdot \frac{g}{(d,\gg)}.
\end{equation*}
L'entier $s(\st)$ divise $d'$, et il est premier à $m'$ d'après le lemme 
\ref{lemsm}~; le résultat s'ensuit. 
\end{proof}

D'après Dipper et James \cite{Dippd2,DJ,James}, si $x\in\X$,
la réduction modulo $\ell$ de $\boldsymbol{\st}(\xt)$ est 
irréductible et cuspidale, et ne dépend que de la partie $\ell$-régulière 
de $\xt$, \ie de l'unique $y\in\mult\kk$ tel que l'ordre de $\xt y^{-1}$ soit 
une puis\-san\-ce de $\ell$.
Ceci définit une application surjective~:
\begin{equation}
\label{james}
y\mapsto\boldsymbol{\s}(y)
\end{equation}
de l'ensemble $\Y$ des 
parties $\ell$-régulières des éléments de $\X$ vers celui des (classes de)
représen\-ta\-tions irréductibles cuspidales $\ell$-modulaires de $\GB$~;
l'ensemble des antécédents de $\boldsymbol{\s}(y)$ est l'orbite de $y$ sous 
le groupe de Galois de $\kk$ sur $\kk_{\D'}$.

\begin{lemm}
Pour $y\in\Y$, soit $\s$ la représentation cuspidale lui
correspondant par \eqref{james}.
On a la relation~: 
\begin{equation}
\label{DEGY}
\deg(y) 
= \frac{m'}{\d(\s)}\cdot\frac{d'}{s(\s)}
\end{equation}
et $s(\s)$ est premier à $m'\d(\s)^{-1}$.
\end{lemm}

\begin{proof}
Si l'on note $m''$ le cardinal de l'orbite de $y$ sous 
$\Gal(\kk/\kk_{\D'})$, alors l'entier $\d(\s)$ défini à la section \ref{SECTION2} 
vérifie la relation $m'=\d(\s)\cdot m''$.
Comme dans le lemme précédent, on en déduit la relation \eqref{DEGY} et que 
$s(\s)$ est premier à $m''=m'\d(\s)^{-1}$.
\end{proof}

Soit $x\in\X$, et soit $y\in\Y$ la partie $\ell$-régulière de $x$.
Soit $\st$ la représentation cus\-pidale $\ell$-adi\-que 
correspondant à $x$ et $\s$ sa réduction modulo $\ell$, qui correspond à $y$.
On pose~:
\begin{eqnarray*}
a(\st) &=& \frac{s(\s)}{s(\st)}, \\
\ww(\st) &=& \frac{\deg(x)}{\deg(y)} 
= \d(\s)a(\st).
\end{eqnarray*}
On a les propriétés suivantes. 

\begin{lemm}
\label{CN}
On a $(\ww(\st),m')=\d(\s)$ et $(\ww(\st),s(\s))=a(\st)$.
\end{lemm}

\begin{proof}
Comme $s(\st)$ est premier à $m'$, il est aussi premier à $\d(\s)$. 
Multipliant par $a(\rt)$, on en déduit que $(\ww(\st),s(\s))=a(\st)$.
Ensuite, $m'\d(\s)^{-1}$ étant premier à $s(\s)$, il est aussi premier à 
$a(\rt)$.
Multipliant par $\d(\s)$, on en déduit que $(\ww(\st),m')=\d(\s)$.
\end{proof}

Notons $\e(\s)$ l'ordre de~:
\begin{equation}
\label{QEKDY}
q_\E^{\deg(y)\cdot\d(\s)}
\end{equation}
dans $(\ZZ/\ell\ZZ)^\times$. 

\begin{lemm}
\label{CN2}
Si $x\neq y$,
le plus grand diviseur de $a(\st)$ premier à $\ell$ est $\e(\s)$.
\end{lemm}

\begin{proof}
Dans $(\ZZ/n\ZZ)^\times$ (où $n$ désigne l'ordre de $x$), 
l'ordre de \eqref{QEKDY} est~:
\begin{equation*}
\frac{\deg(x)}{(\deg(y) \d(\s),\deg(x))}=a(\st).
\end{equation*}
Comme $x\neq y$, l'entier $n$ est divisible par $\ell$.
En projetant sur $(\ZZ/\ell\ZZ)^\times$, on en déduit que le plus grand 
diviseur de $a(\st)$ premier à $\ell$ est $\e(\s)$.
\end{proof}

\section{Comptage}

\subsection{Preuve du théorème \ref{CongCusp}}

Soit $\rt$ une représentation irréductible cuspidale $\ell$-adique entière de $\G$, 
et soit $\Oo(\rt)$ l'ensemble des classes inertielles $[\G,\rt']$ 
de représentations irréductibles cuspi\-dales $\ell$-adiques de $\G$ 
congrues à $[\G,\rt]$ modulo $\ell$, \ie telles que~:
\begin{equation*}
\r_\ell([\G,\rt'])=\r_\ell([\G,\rt]).
\end{equation*}
Fixons un type simple maxi\-mal $(\J,\lt)$ dans la classe de $\G$-conjugaison 
correspondant à $[\G,\rt]$,
et notons $\l$ la réduction de $\lt$ modulo $\ell$.
Alors $(\J,\l)$ est un $\flb$-type simple maximal correspondant à la classe inertielle 
de la représentation $\rho$ apparaissant dans \eqref{redcusp}, et l'entier 
$a=a(\rt)$ est l'indice du $\G$-normalisateur de $(\J,\lt)$ dans celui de 
$(\J,\l)$
(voir \cite[\S3]{MSt}).

L'ensemble $\Oo(\rt)$ s'identifie donc à l'ensemble des classes de
$\G$-conjugaison de $\qlb$-types simples maximaux $(\J',\lt')$ 
tels que, si l'on note $\l'$ la réduction de $\lt'$ modulo $\ell$, on ait~:
\begin{enumerate}
\item 
les $\flb$-types simples maximaux $(\J',\l')$ et $(\J,\l)$ sont conjugués sous
$\G$~; 
\item
on a $(\N_\G(\J',\l'):\N_\G(\J',\lt'))=(\N_\G(\J,\l):\N_\G(\J,\lt))$~;
\end{enumerate}
où $\N_\G(\J,\l)$ désigne le normalisateur de $(\J,\l)$ dans $\G$.
Quitte à conjuguer,
on peut donc supposer que $\J'=\J$ et $\l'=\l$~; 
l'ensemble $\Oo(\rt)$ s'identifie donc à l'ensemble $\Tt(\J,\lt)$ des classes de 
$\N_\G(\J,\l)$-conjugaison de $\qlb$-types simples maximaux $(\J,\lt')$ de $\G$ 
tels que~:
\begin{enumerate}
\item
les représentations $\lt'$ et $\lt$ sont congruentes modulo $\ell$~;
\item
les paires $(\J,\lt')$ et $(\J,\lt)$ ont le même normalisateur dans $\G$.
\end{enumerate}
Fixons une décomposition de $\lt$ sous la forme $\kt\otimes\st$
et un isomorphisme de groupes de $\J/\J^1$ sur $\GB$
(voir la section \ref{SECTION2}). 
Le foncteur~:
\begin{equation*}
\widetilde{\tau}\mapsto\kt\otimes\widetilde{\tau}
\end{equation*}
définit une bijection entre les représentations irréducti\-bles 
cus\-pidales de $\GB$ et les 
ty\-pes simples maxi\-maux de $\G$ définis sur $\J$ et contenant $\kt$.
D'après \cite[Theorem 7.2]{SeSt2},
sa réciproque induit une bijection de $\Tt(\J,\lt)$ sur l'en\-semble $\Cc(\st)$ des 
orbites, sous l'action du groupe de Galois $\Gal(\kk_{\D'}/\kk_{\E})$, 
de représentations irréducti\-bles cus\-pidales $\st'$
de $\GB$ telles que~: 
\begin{enumerate}
\item
les représentations $\st'$ et $\st$ sont congruentes modulo $\ell$~;
\item 
les orbites de $\st'$ et de $\st$ sous $\Gal(\kk_{\D'}/\kk_{\E})$ ont le même cardinal.
\end{enumerate}
Si $\xt\in\X$ correspond à $\st$, alors \eqref{green} 
induit une bijection entre $\Cc(\st)$ et l'ensemble $\EuScript{K}(\xt)$ 
des orbites des éléments $\xt'\in\X$, sous le groupe de 
Galois $\Ga$ de $\kk$ sur $\kk_\E$, tels que~:
\begin{enumerate}
\item
$x'$ et $x$ ont la même partie $\ell$-régulière~;
\item
les $\Ga$-orbites de $x'$ et de $x$ ont le même cardinal,
\ie que $\deg(\xt')=\deg(x)$.
\end{enumerate}
Remarquons que la condition 2 ci-dessus signifie que $\xt$, $\xt'$ ont le 
même stabilisateur dans $\Ga$.
En prenant l'intersection avec $\Gal(\kk/\kk_{\D'})$, on voit que tout 
$x'\in\mult\kk$ vérifiant la condition 2 appartient automatiquement à $\X$.
On obtient finalement une bijection entre $\Oo(\rt)$ et $\EuScript{K}(\xt)$~; 
on a donc prouvé le résul\-tat suivant. 

\begin{prop}
L'ensemble $\Oo(\rt)$ est fini, et son cardinal $t(\rt)$ est 
le nombre de $\Ga$-orbites des $\xt'\in\mult\kk$ tels que 
$x,x'$ ont la même partie $\ell$-régulière et le même degré 
sur $\kk_\E$.
\end{prop}

Écrivons $\xt$ sous la forme $yz$ où $y$ est d'ordre premier à $\ell$ et
$z$ d'ordre une puissance de $\ell$
(donc $y$ est la partie $\ell$-régulière de $x$).
L'application~:
\begin{equation}
\label{phichi}
z'\mapsto yz'
\end{equation}
est une bijection entre la composante $\ell$-primaire $\P_{\ell}$ de $\mult\kk$ 
et l'ensemble des $x'\in\mult\kk$ dont la partie $\ell$-régulière est $y$.
Étant donnés $z'\in\P_{\ell}$ et $k\in\ZZ$, remarquons que~:
\begin{equation}
\label{ponsonduterrail}
(yz')^{q_{\E}^{k}}=yz'
\quad\Leftrightarrow\quad 
\text{$y^{q_{\E}^{k}}=y$ et $(z')^{q_{\E}^{k}}=z'$}.
\end{equation}
Notons $\kk_{1}$ l'extension de $\kk_\E$ engendrée par $y$.
Pour $z'\in\P_{\ell}$, notons $[\![z']\!]$ son orbite sous le 
groupe de Galois de $\kk$ sur $\kk_{1}$ et~:
\begin{equation*}
\deg_{1}(z') = {\rm card}([\![z']\!])
\end{equation*}
son degré sur $\kk_{1}$.
D'après \eqref{ponsonduterrail}, les éléments $yz$, $yz'$ ont le même 
degré sur $\kk_\E$ si et seulement si~:
\begin{equation}
\label{Bateman}
\deg_{1}(z) = \deg_{1}(z').
\end{equation}
Notons $\Pp(z)$ l'ensemble des $[\![z']\!]$ 
pour $z'\in\P_\ell$ vérifiant \eqref{Bateman}.
On a prouvé le résultat suivant. 

\begin{lemm}
L'application \eqref{phichi} induit une bijection de $\Pp(z)$ sur 
$\EuScript{K}(\xt)$.  
\end{lemm}

Il ne nous reste plus qu'à calculer $\deg_{1}(z)$ en 
fonction des invariants associés à $\rt$.
Comme on a $a(\rt)=a(\st)$ et $\d(\rho)=\d(\s)$, et compte tenu de 
\eqref{DEGX} et \eqref{DEGY}, on en déduit que~:
\begin{equation}
\label{DEGZ}
\deg_{1}(z) = 
\frac{\deg(x)}{\deg(y)} = 
\d(\rho) a(\rt)
\end{equation}
que l'on note $\ww(\rt)=\ww(\st)$.
On obtient donc le résultat suivant. 

\begin{lemm}
\label{L3}
L'entier $\ww(\rt)t(\rt)$ est le nombre de $z'\in\P_\ell$ de degré $\ww(\rt)$ 
sur $\kk_{1}$.
\end{lemm}

Compte tenu de la relation $n(\rt)s(\rt)=f(\rt)$ (voir la section 
\ref{SECTION2}), l'extension de $\kk_{1}$ de degré $\ww(\rt)$ 
est de cardinal~: 
\begin{equation*}
q^{n(\rt)}.
\end{equation*}
On en déduit l'inégalité $t(\rt)\<\cc(\rt)$, et cette inégalité est une 
égalité si et seulement si $\ww(\rt)=1$, \ie si et seulement si $\rt$ est 
$\ell$-supercuspidale. 
Ceci met fin à la preuve du théorème \ref{CongCusp}.

\subsection{Preuve du théorème \ref{CalculW}}

Poussons maintenant plus loin les calculs dans le cas où 
$\r_\ell(\rt)$ n'est pas irréductible et supercuspidale, \ie que 
$\ww=\ww(\rt)>1$. 
Notons $\Q$ le cardinal de $\kk_{1}$ et, pour tout $n\>1$,
notons $f(n)=f_\Q(n)$ le nombre de $z'\in\P_\ell$ de degré $n$ 
sur $\kk_{1}$.
D'après le lemme \ref{L3}, on a donc~:
\begin{equation}
\label{memento}
t(\rt)=\frac{f(\ww)}{\ww}.
\end{equation}
Notons $v$ la valuation $\ell$-adique sur $\ZZ$, 
et notons $\kk'$ l'extension de $\kk_{1}$ de degré 
$w$ contenue dans $\kk$~; 
en partitionnant $\kk^{\prime\times}$ 
selon le degré de ses élément sur $\kk_{1}$, on obtient l'égalité~:
\begin{equation*}
\ell^{v(\Q^{\ww}-1)} = \sum\limits_{n |\ww} f(n).
\end{equation*}
Par inversion de Möbius, on a~:
\begin{equation*}
f(\ww) = \sum\limits_{n | \ww} 
\mu\left(\frac{\ww}{n}\right)
\ell^{v(\Q^{n}-1)}
\end{equation*}
où $\mu$ désigne la fonction de Möbius. 
Notons $w_0$ le plus grand diviseur de $\ww$ premier à $\ell$.

\begin{lemm}
L'ordre de $\Q$ dans $(\ZZ/\ell\ZZ)^\times$ est égal à $w_0$.
\end{lemm}

\begin{proof}
L'ordre de $z$ est de la forme $\ell^r$, $r\>0$.
Comme $\ww>1$, on déduit que $r\>1$.
La condition $\deg_{1}(z)=\ww$ signifie que l'ordre de $\Q$ 
dans $(\ZZ/\ell^{r}\ZZ)^\times$ est égal à $\ww$.
En projetant sur $(\ZZ/\ell\ZZ)^\times$, on en déduit que l'ordre de $\Q$ 
dans $(\ZZ/\ell\ZZ)^\times$ est égal au plus grand diviseur de $\ww$ 
premier à $\ell$, \ie $w_0$.
\end{proof}

On a donc~:
\begin{equation*}
f(\ww) = \sum\limits_{t\<v(\ww)} \sum\limits_{n | w_0} \ 
\mu(\ell^{v(\ww)-t})\mu\left(\frac{w_0}{n}\right)
\ell^{v(\Q^{n\ell^t}-1)}.
\end{equation*}
Si $v(\ww)=0$, on a $\ww=w_0>1$ et cela donne simplement~:
\begin{equation*}
f(w_0) = \sum\limits_{n | w_0} 
\mu\left(\frac{w_0}{n}\right)
\ell^{v(\Q^{n}-1)}.
\end{equation*}
On trouve que~:
\begin{equation*}
f(w_0) = \ell^{v(\Q^{w_0}-1)} + \sum\limits_{n | w_0 \atop{n\neq w_0}}
\mu\left(\frac{w_0}{n}\right)
=\ell^{v(\Q^{w_0}-1)}-1.
\end{equation*}
Supposons maintenant que $v(\ww)\>1$. 
Cela donne~:
\begin{equation*}
f(\ww) = \sum\limits_{n | w_0} \mu\left(\frac{w_0}{n}\right) \ell^{v(\Q^{n\ell^{v(\ww)}}-1)}
- \sum\limits_{n | w_0} 
\mu\left(\frac{w_0}{n}\right)\ell^{v(\Q^{n\ell^{v(\ww)-1}}-1)} \\ 
= f_{\Q^{\ell^{v(\ww)}}}(w_0)-f_{\Q^{\ell^{v(\ww)-1}}}(w_0).
\end{equation*}
Comme $\Q$ a le même ordre que $\Q^{\ell^k}$, $k\>0$,
dans $(\ZZ/\ell\ZZ)^\times$, à savoir $w_0$, 
on trouve que~:
\begin{equation*}
f(w) = \ell^{v(\Q^{u\ell^v}-1)}-\ell^{v(\Q^{u\ell^{v-1}}-1)}
= \ell^{v(\Q^{\ww}-1)-1}(\ell-1).
\end{equation*}
On trouve ainsi le résultat annoncé, en remarquant que $c(\rt)$ est égal à 
$\ell^{v(\Q^{\ww}-1)}$.


\subsection{Lien avec la formulation de Vignéras et Dat}

Dans ce paragraphe, nous allons reformuler le théorème \ref{CongCusp} sous une 
forme analogue à celles de Vignéras \cite[Proposition 2.3]{Vigl} et Dat 
\cite[Proposition 2.3.2]{Datj}.

Fixons une uniformisante $\w$ de $\F$ et,
pour toute représentation irréductible cuspidale $\ell$-adique entiè\-re $\rt$ 
de $\G$, 
notons $\Oo(\rt,\w)$ l'en\-semble des classes de représentations 
irréductibles cuspidales $\ell$-adiques entiè\-res de $\G$ qui sont 
congrues à $\rt$ et dont le caractère central prend la même valeur que celui 
de $\rt$ en $\w$.
Soit $\C(\rt)$ la plus grande puissance de $\ell$ divisant~:
\begin{equation*}
\frac{md}{n(\rt)}\cdot(q^{n(\rt)}-1).
\end{equation*}
On a le résultat suivant. 

\begin{prop}
\label{CongCusp2}
Soit $\rt$ une $\qlb$-représentation irréductible cuspidale et entière de $\G$. 
Alors l'en\-semble $\Oo(\rt,\w)$ est fini, de cardinal noté $\T(\rt)$,
et on a~:
\begin{equation*}
\T(\rt)\<\C(\rt)
\end{equation*}
avec égalité si et seulement si $\rt$ est $\ell$-supercuspidale.  
\end{prop}

\begin{proof}
D'après le théorème \ref{CongCusp},
il suffit de prouver que $\T(\rt)$ est le produit de 
$t(\rt)$ par la plus grande puissance de $\ell$ divisant
$mdn(\rt)^{-1}$.
Commençons par remplacer la cor\-respon\-dance bijective \eqref{sichuan} 
par la bijection~:
\begin{equation*}
\rho \leftrightarrow (\boldsymbol{\J},\Lambda)
\end{equation*}
entre classes d'isomorphisme de représentations irréductibles cuspidales 
de $\G$ et clas\-ses de conju\-gai\-son (sous $\G$) de types simples 
maximaux étendus de $\G$ (voir \cite[Théorème 3.11]{MSt}).

Soit $(\J,\lt)$ un type simple maximal contenu dans la 
$\qlb$-représentation irréductible cuspidale $\rt$, 
et soit $\widetilde{\boldsymbol{\J}}$ son normalisateur dans $\G$. 
D'après \cite[Proposition 3.1]{MSt}, il y a une unique représentation 
$\tL$ de $\widetilde{\boldsymbol{\J}}$ prolongeant $\lt$ dont l'induite à $\G$ est 
isomorphe à $\rt$.
Notons $\overline\Lambda$ la réduction modulo $\ell$ de $\tL$, 
qui est un prolongement de $\l$ à $\widetilde{\boldsymbol{\J}}$.

Soit $\rt'$ une représentation irréductible cuspidale $\ell$-adique entière de 
$\G$.  
Pour qu'elle soit con\-grue à $\rt$, il faut et il suffit qu'elle contienne un 
type simple maximal étendu $(\widetilde{\boldsymbol{\J}}{}',\tL')$ tel que 
$\widetilde{\boldsymbol{\J}}{}'$ soit égal à $\widetilde{\boldsymbol{\J}}$ 
et dont la réduction modulo 
$\ell$, notée $\overline\Lambda{}'$, soit égale à $\overline\Lambda$.
L'entier $\T(\rt)$ est donc le nombre de classes de $\G$-conjugaison 
de $(\widetilde{\boldsymbol{\J}}{}',\tL')$ tels que~:
\begin{equation*}
\widetilde{\boldsymbol{\J}}{}'=\widetilde{\boldsymbol{\J}},
\quad
\overline\Lambda{}'=\overline\Lambda
\quad\text{et}\quad
\tL'(\w)=\tL(\w).
\end{equation*}
Fixons une uniformisante $\w'$ de $\D'$ et posons~:
\begin{equation*}
\widetilde\w=(\w')^{d's(\rt)^{-1}}.
\end{equation*}
Le groupe 
$\widetilde{\boldsymbol{\J}}$ est engendré par $\J$ et $\widetilde\w$.
L'entier $\T(\rt)$ est égal au produit $t(\rt)t_1(\rt)$ 
où $t_1(\rt)$ est le nombre de re\-pré\-sen\-tations $\tL'$ de 
$\widetilde{\boldsymbol{\J}}$ 
prolongeant $\lt$ telles que~:
\begin{equation*}
\overline\Lambda{}'(\widetilde\w)=\overline\Lambda(\widetilde\w)
\quad\text{et}\quad
\tL'(\w)=\tL(\w).
\end{equation*}
Le nombre de re\-pré\-sen\-tations irréductibles de $\widetilde{\boldsymbol{\J}}$ prolongeant 
$\l$ et prenant une 
valeur fixée en $\w$ est égal à l'indice de $\mult\F\boldsymbol{\J}$ dans $\widetilde{\boldsymbol{\J}}$, 
\ie à~:
\begin{equation}
\label{T1}
e(\E:\F)s(\rt)=\frac{md}{n(\rt)},
\end{equation}
où $e(\E:\F)$ désigne l'indice de ramification de $\E$ sur $\F$.
Compte tenu de la condition supplémen\-tai\-re sur 
$\overline\Lambda{}'(\widetilde\w)$, on trouve 
que $t_1(\rt)$ est la plus grande puissance de $\ell$ divisant \eqref{T1}.
\end{proof}

\section{Preuve du théorème \ref{lift}}

Soit $\rho$ une représentation irréductible cuspidale $\ell$-modulaire de 
$\G$ et soit $(\J,\k\otimes\s)$ un $\flb$-type simple maxi\-mal dans la classe de 
$\G$-conjugaison correspondant à $[\G,\rho]$.
D'après \cite{MSt}, pour que $\rho$ se relève à $\qlb$, il faut et 
suffit que $\s$, considérée comme une représentation irréductible cuspidale 
de $\GB$, se relève en une représentation 
irréductible cuspidale $\ell$-adique $\st$ telle que $s(\st)=s(\s)$.

Soit $y\in\Y$ correspondant à $\s$ par \eqref{james}.  
Pour qu'une telle représentation $\st$ existe, 
il faut et il suffit donc, d'après \eqref{DEGZ},
qu'il existe un $x\in\X$ dont la partie 
$\ell$-régulière soit $y$ et qui vérifie~:
\begin{equation*}
\deg(x) = \d(\s)\cdot\deg(y).
\end{equation*}

Si $\rho$ (donc $\s$) est supercuspidale, \ie si l'on a $\d(\s)=1$,
alors $x=y\in\X$ vérifie les conditions requises, et on retrouve bien le 
fait que toute représentation irréductible supercuspidale $\ell$-modulaire 
se relève. 

Supposons maintenant que $\rho$ est cuspidale mais pas 
supercuspidale, \ie que $\d(\s)>1$.
Notons~:
\begin{equation*}
\e(\rho) 
\end{equation*}
l'ordre de $q^{n(\rho)}$ dans $(\ZZ/\ell\ZZ)^\times$
(voir le paragraphe \ref{PARA21} pour la définition de l'entier $n(\rho)$). 
Soit $\nu$ le ca\-rac\-tère non ramifié $\ell$-modulaire
de $\G$ obtenu en com\-po\-sant la norme 
réduite de $\G$ sur $\mult\F$, la va\-lua\-tion de $\mult\F$ dans $\ZZ$ 
(envoyant une uniformisante sur $1$) et le morphisme 
envoyant $1$ sur l'in\-ver\-se de $q$ modulo $\ell$.

\begin{lemm}
\label{STEP0}
Soit un entier $i\in\ZZ$.
Pour que $\rho\nu^i=\rho$, il faut et il suffit que $\e(\rho)$ 
divise $i$.
\end{lemm}

\begin{proof}
D'après \cite[\S3.4]{MSt}, 
les représentations $\rho\nu^i$ et $\rho$ sont isomorphes 
si et seulement si $\nu^{n(\rho)i}=1$. 
L'ordre de $\nu$ étant égal à 
l'ordre de $q$ dans $(\ZZ/\ell\ZZ)^\times$, noté $e$, 
ceci équivaut à dire que $e$ divise $n(\rho)i$.
Il ne reste plus qu'à remarquer que~:
\begin{equation*}
\e(\rho) = \frac{e}{(e,n(\rho))}
\end{equation*}
pour conclure. 
\end{proof}

\begin{coro}
On a $\rho\nu\simeq\rho$ si et seulement si $\e(\rho)=1$.
\end{coro}


Ainsi le théorème \ref{lift} peut être reformulé de la façon suivante. 

\begin{theo}
\label{lift2}
Soit $\rho$ une représentation irréductible cuspidale non 
supercuspidale $\ell$-modu\-lai\-re de $\G$.
Pour que $\rho$ se relève à $\qlb$, il faut et il suffit que 
les entiers $s(\rho)$ et $\d(\rho)$ soient premiers entre eux
et que $\e(\rho)=1$.
\end{theo}

D'après le paragraphe \ref{PARA21}, l'entier $n(\rho)$ est 
le plus grand diviseur de $f(\rho)s(\rho)^{-1}$ premier à $\ell$.
Par conséquent, $\e(\rho)$ est égal à l'entier $\e(\s)$ du paragraphe 
\ref{PARA22}. 

\begin{lemm}
\label{CN3}
Pour toute représentation irréductible cuspidale $\ell$-adique $\st$ relevant 
$\s$, le plus grand diviseur de $a(\st)$ premier à $\ell$ est $\e(\s)$.
\end{lemm}

\begin{proof}
Fixons un $x\in\X$ correspondant à $\st$ et de partie régulière $y$.
Comme $\rho$ (donc $\s$) n'est pas supercuspidale,
on a $x\neq y$.
Le résultat suit alors du lemme \ref{CN2}.
\end{proof}

Posons $f(\s)=f(\rho)$ (voir le paragraphe \ref{PARA21}).

\begin{lemm}
\label{LEM1}
Soit $z\in\P_\ell$ d'ordre $\ell^r$, $r\>0$.
On a $yz\in\X$ si et seulement si l'ordre de~:
\begin{equation}
\label{ORDREMYPROGRESSO}
q^{f(\s)\d(\s)^{-1}}
\end{equation}
dans $(\ZZ/\ell^r\ZZ)^\times$ est égal à $\d(\s)$.
\end{lemm}

Comme $y\in\Y$, il y a un $z\in\P_\ell$ (non trivial puisque $\s$ n'est pas 
supercuspidale) tel que $yz\in\X$.
Il y a donc un $r\>1$ tel que l'ordre de \eqref{ORDREMYPROGRESSO}
dans $(\ZZ/\ell^r\ZZ)^\times$ est égal à $\d(\s)$. 
En réduisant modulo $\ell$, on en déduit que son ordre dans 
$(\ZZ/\ell\ZZ)^\times$ est le plus grand diviseur de $\d(\s)$ 
premier à $\ell$.

\begin{lemm}
\label{LEM2}
Soit $z\in\P_\ell$ d'ordre $\ell^r$, $r\>0$.
On a $\deg(yz)=\d(\s)\cdot\deg(y)$ si et seulement si 
l'ordre de~:
\begin{equation}
\label{ORDREMYPROGRESSO2}
q^{f(\s)(\d(\s)s(\s))^{-1}}
\end{equation} 
dans $(\ZZ/\ell^r\ZZ)^\times$ est égal à $\d(\s)$.
\end{lemm}

Supposons d'abord que $\rho$ se relève à $\qlb$.
D'après les lemmes \ref{CN} et \ref{CN2}, on trouve que $\d(\rho)$ est premier à 
$s(\rho)$ et que $\e(\rho)=1$.

Inversement, supposons que les conditions du théorème \ref{lift} sont vérifiées. 
Soit $z\in\P_\ell$ d'ordre $\ell^r$, $r\>1$, tel que $yz\in\X$, et notons $n$ 
l'ordre de \eqref{ORDREMYPROGRESSO2} dans $(\ZZ/\ell^r\ZZ)^\times$.
D'après le lemme \ref{LEM1}, on a~:
\begin{equation}
\label{CRUX}
\frac{n}{(n,s(\s))} = \d(\s).
\end{equation}
L'hypothèse $\e(\rho)=1$ implique que~:
\begin{equation}
\label{CRUX2}
\frac{n}{(n,\d(\s))} = \ell^{t},
\quad t\>0.
\end{equation}
Si $\ell$ divise $\d(\s)$, alors $s(\s)$ est premier à $\ell$, et \eqref{CRUX} 
et \eqref{CRUX2} impliquent que $n=\d(\s)$.

En revan\-che, si $\d(\s)$ est premier à $\ell$, écrivons $n=\d(\s)n'$ avec 
$n'=(n,s(\s))=\ell^t$.
On peut remarquer que $t=v(n)$.
Alors l'élément~:
\begin{equation*}
z^{\ell^{v(n)}}\in\P_\ell
\end{equation*}
qui est d'ordre $\ell^{r-v(n)}$, 
vérifie la condition du lemme \ref{LEM2} car l'ordre de 
\eqref{ORDREMYPROGRESSO2} dans $(\ZZ/\ell^{r-v(n)}\ZZ)^\times$ est égal à 
$n\ell^{-v(n)}=\d(\s)$. 
Comme $\d(\rho)$ est premier à $s(\rho)$, il vérifie aussi la condition du 
lemme \ref{LEM1}. 
Ceci met fin à la preuve du théorème \ref{lift}.

\begin{rema}
\label{Soitditenpassant}
Posons $k=k(\s)$ et $s=s(\s)$,
et notons $\tau$ l'unique représentation irré\-ducti\-ble
super\-cuspidale de $\GL_{m'k^{-1}}(\kk_{\D'})$ telle que $\s$ soit un 
sous-quotient de l'induite parabolique de 
$\tau\otimes\dots\otimes\tau$. 
Le plus grand diviseur de $k$ premier à $\ell$ est
$\e(\tau)(s,\e(\tau))^{-1}$ et $\e(\s)$ est égal à $(s,\e(\tau))$.
La condition du théorème \ref{lift} s'écrit donc $\e(\rho)=1$ et 
min$(v(k),v(s))=0$.
\end{rema}

\section{Preuve du théorème \ref{lifta}}
\label{SECTION5}

Soit $\rho$ une représentation irréductible cuspidale $\ell$-modulaire de 
$\G$, et soit $a>1$.
On cherche à quelle condition $\rho$ admet un $a$-relèvement, \ie une 
représentation irréductible cuspidale $\ell$-adique entière $\rt$ de $\G$ dont la 
réduc\-tion modulo $\ell$ contienne $\rho$ et soit de longueur $a$. 
Il faut et suffit pour cela que $\s$ se relève en une représentation 
irréductible cuspidale $\ell$-adique $\st$ telle que $s(\s)=a\cdot s(\st)$.
D'après les lemmes \ref{CN} et \ref{CN2}, on a des conditions nécessaires~: 
\begin{enumerate}
\item
$a=\e(\s)\ell^u$ avec $u\>0$.
\item
$a$ divise $s(\s)$ et $s(\s)a^{-1}$ est premier à $\d(\s)$.
\end{enumerate}
Supposons donc qu'elles sont vérifiées.

\begin{rema}
En particulier, si $\s$ n'est pas supercuspidale,
les lemmes \ref{CN} et \ref{CN2} montrent que 
$\e(\s)$ divise $s(\s)$.
\end{rema}

Soit $y\in\Y$ correspondant à $\s$ par \eqref{james}.  
Pour qu'une telle représentation $\st$ existe, 
il faut et il suffit donc, d'après \eqref{DEGZ},
qu'il existe un $x\in\X$ dont la partie 
$\ell$-régulière soit $y$ et qui vérifie~:
\begin{equation*}
\deg(x) = a\d(\s)\cdot\deg(y).
\end{equation*}
Comme $y\in\Y$, il y a un $z\in\P_\ell$ (non trivial car $a>1$) 
tel que $yz\in\X$.
Il y a donc un entier $r\>1$ tel que l'ordre de \eqref{ORDREMYPROGRESSO}
dans $(\ZZ/\ell^r\ZZ)^\times$ est égal à $\d(\s)$. 
En réduisant modulo $\ell$, on en déduit que son ordre dans 
$(\ZZ/\ell\ZZ)^\times$ est le plus grand diviseur de $\d(\s)$ 
premier à $\ell$.

\begin{lemm}
\label{LEM2a}
Soit $z\in\P_\ell$ d'ordre $\ell^r$, $r\>0$.
On a $\deg(yz)=a\d(\s)\cdot\deg(y)$ si et seulement si 
l'ordre de~:
\begin{equation}
\label{ORDREMYPROGRESSO2a}
q^{f(\s)(\d(\s)s(\s))^{-1}}
\end{equation} 
dans $(\ZZ/\ell^r\ZZ)^\times$ est égal à $a\d(\s)$.
\end{lemm}

Soit $z\in\P_\ell$ d'ordre $\ell^r$, $r\>1$, tel que $yz\in\X$, et soit $n$ 
l'ordre de \eqref{ORDREMYPROGRESSO2a} dans $(\ZZ/\ell^r\ZZ)^\times$.
D'après le lemme \ref{LEM1}, on a~:
\begin{equation}
\label{CRUXa}
\frac{n}{(n,s(\s))} = \d(\s).
\end{equation}
En particulier, on a~:
\begin{equation}
\label{CRUXabis}
v(n)=v(k)+{\rm min}(v(n),v(s(\s))).
\end{equation}
Si l'on note $n_0$ le plus grand diviseur de $n$ premier à $\ell$, on a~:
\begin{equation}
\label{CRUX2a} 
\frac{n_0}{(n_0,\d(\s))} = \e(\s).
\end{equation}
On déduit de l'hypothèse sur $a$ que $(n_0,s(\s))=\e(\s)$, 
puis que $n_0=\d_0(\s)\e(\s)$, 
où $\d_0(\s)$ désigne le plus grand diviseur de 
$\d(\s)$ premier à $\ell$. 

On cherche un $t\in\{1,\dots,v(q^{f(\s)}-1)\}$ 
tel que l'ordre de \eqref{ORDREMYPROGRESSO2a} dans $(\ZZ/\ell^t\ZZ)^\times$ 
soit égal à $a\d(\s)$.
D'après l'hypo\-thèse sur $a$, cela impliquera automatiquement que 
l'ordre de \eqref{ORDREMYPROGRESSO} est $\d(\s)$. 
Soit~:
\begin{equation*}
t_0 = v(\Q^{n_0}-1).
\end{equation*}
On a donc~:
\begin{equation*}
v(q^{f(\s)}-1) 
= t_0 + v(s(\s)) + v(k).
\end{equation*}
Posons $t=t_0+u+v(k)$ (on a bien $t\<v(q^{f(\s)}-1)$ car $u\<v(s(\s))$).
Alors l'ordre de \eqref{ORDREMYPROGRESSO2a} 
dans $(\ZZ/\ell^t\ZZ)^\times$ est égal à $n_0\ell^u=a\d(\s)$. 

\begin{rema}
Compte tenu de la remarque \ref{Soitditenpassant}, la condition du théorème \ref{lifta} se 
résume à $u\in\{0,\dots,v(s(\rho))\}$ et min$(v(k(\rho)),v(s(\rho))-u)=0$.
\end{rema}

\section{Preuve du théorème \ref{comptagedeclasses}}

Dans toute cette section, on fixe des entiers $n,w\>1$ tels que $w$ divise 
$n$.

\subsection{}
\label{ENDO}

Dans ce paragraphe, nous rappelons brièvement quelques 
attributs des $\F$-endoclasses de carac\-tè\-res simples
dont nous aurons besoin.
Pour plus de détails, nous renvoyons le lecteur à \cite{BSS}. 

Soit $\A$ une $\F$-algèbre centrale simple de dimension finie,
soit $[\La,n_\La,0,\b]$ une strate simple de $\A$ et soit 
$\t\in\Cc(\La,0,\b)$ un caractère simple.
Le couple $([\La,n_\La,0,\b],\t)$ 
défi\-nit un ps-caractère dont l'endo-classe 
-- qui est une classe d'équivalence de ps-ca\-rac\-tères --
est notée $\TT$.
Les entiers~:
\begin{eqnarray*}
f(\TT) &=& f(\F[\b]:\F), \\
\deg(\TT) &=& [\F[\b]:\F],
\end{eqnarray*} 
\ie le degré résiduel et le degré 
de $\F[\b]$ sur $\F$ respectivement, 
ne dépendent pas du choix de $\b$ mais uniquement de $\TT$.
Le nombre rationnel positif~:
\begin{equation*}
l(\TT) = -v_{\F}(\b),
\end{equation*}
où $v_{\F}$ désigne la valuation sur $\F[\b]$ normalisée 
en donnant la valeur $1$ à une uniformisante de $\F$, ne dépend pas non plus 
du choix de $\b$ mais uniquement de $\TT$.

Si $\rho$ est une $\R$-représentation irréductible cuspidale de $\mult\A$ 
(ici $\R$ désigne $\qlb$ ou $\flb$), il existe un couple 
$([\La,n_\La,0,\b],\t)$ comme ci-dessus tel que la 
restriction de $\rho$ au pro-$p$-sous-groupe $\H^1(\b,\La)$
contienne $\t$.
L'endoclasse $\TT$ définie par ce couple ne dépend que de la classe d'isomorphisme 
de la représentation $\rho$,
et le nombre rationnel $l(\TT)\>0$ est le niveau normalisé 
(ou aussi profon\-deur) de $\rho$.

\subsection{}

Soit $\D$ une $\F$-algèbre centrale simple de degré réduit $d$ divisant $n$,
et soit $\TT$ une $\F$-endoclasse de degré $\gg$ divisant $nw^{-1}$.
On définit un entier $m\>1$ par la relation $md=n$ et on pose~:
\begin{equation*}
d' = \frac{d}{(d,\gg)},
\quad
m' = \frac{m(d,\gg)}{\gg}.
\end{equation*}
Pour tout $u\>1$ divisant $m$, 
notons $\Aa(\D,\TT,w,u)$ l'en\-sem\-ble des classes inertielles de 
re\-pré\-sen\-ta\-tions 
irréductibles cuspidales $\ell$-adiques $\rt$ de $\GL_{u}(\D)$ telles que~: 
\begin{enumerate}
\item
$w(\rt)=w$~;
\item
l'endoclasse de $\rt$ est égale à $\TT$.
\end{enumerate}
Remarquons que, pour qu'il y ait une représentation 
irréductible cuspidale $\ell$-adique $\rt$ de $\GL_{u}(\D)$ 
d'endoclasse $\TT$, il faut et il suffit que l'entier $u$ soit 
de la forme~:
\begin{equation} 
\label{quark}
u = \frac{\gg}{(d,\gg)}\cdot u'
\end{equation}
où $u'$ est un diviseur de $m'$.
Posons maintenant~:
\begin{equation}
\label{Pietranera}
\Aa(\D,\TT,w) = 
\bigcup\limits_{u|m} \Aa(\D,\TT,w,u)
\end{equation}
et notons $\Aa_{\ell}(\D,\TT,w)$ l'ensemble des réductions mod $\ell$ 
des éléments de \eqref{Pietranera}. 
L'endoclasse $\TT$ étant fixée, cet ensemble est 
fini, et son cardinal sera noté $\boldsymbol{a}_{\ell}(\D,\TT,w)$.

\subsection{}

Dans ce paragraphe, on suppose que $\TT$ est la $\F$-endoclasse nulle $\0$, 
et on va calculer $\boldsymbol{a}_{\ell}(\D,\0,w)$.
Étant donné un entier $u\>1$ divisant $m$,
toute représentation irréductible cuspidale $\ell$-mo\-dulaire $\s$ de $\GL_u(\kk_\D)$ 
définit un type simple maximal de niveau $0$ de $\GL_u(\D)$ |
donc une classe inertielle $\Omega(\s)$ de représentations 
cuspidales de niveau $0$ de $\GL_u(\D)$.
L'application~:
\begin{equation}
\label{Gina}
\s \mapsto \Omega(\s)
\end{equation}
est surjective (toutes les classes inertielles de représentations 
cuspidales de niveau $0$ de $\GL_u(\kk_\D)$ sont atteintes) et ses 
fibres sont les classes de conjugaison sous le groupe de Galois de 
$\kk_{\D}$ sur $\kk_\F$.

\begin{lemm}
\label{riddle}
Soit $\s$ une représentation irréductible cuspidale $\ell$-mo\-dulaire de 
$\GL_u(\kk_\D)$. 
Alors $\Omega(\s)$ appartient à $\Aa_\ell(\D,\0,w)$ si et seulement 
s'il existe une représentation irréductible cuspidale $\ell$-adique 
$\st$ de $\GL_u(\kk_\D)$ dont la réduction modulo $\ell$ soit $\s$ et 
telle que $w(\st)=w$.
\end{lemm}

Notant $\BB_{\ell}(q,m,d,w)$ l'image réciproque de $\Aa_\ell(\D,\0,w)$ par 
\eqref{Gina}, qui est décrite par le lemme \ref{riddle}, on obtient ainsi~:
\begin{equation*}
\boldsymbol{a}_{\ell}(\D,\0,w) = \sum_{\s} \frac{s(\s)}{d}
\end{equation*}
où $\s$ décrit l'ensemble $\BB_{\ell}(q,m,d,w)$.

Fixons une clôture algébrique $\overline{\kk}$ de $\kk_{\D}$ et notons 
$\kk$ l'extension de $\kk_\F$ de degré $nw^{-1}$ incluse dans 
$\overline{\kk}$. 
(Attention~! Cette extension $\kk$ de $\kk_\F$ 
ne coïncide avec l'extension ainsi notée 
dans la section \ref{SECTION2} que si $w=1$.)
Pour tout $y\in\kk^{\times}$, posons~:
\begin{eqnarray*}
\deg(y) &=& \text{degré de $y$ sur $\kk_\F$}, \\
\ee(y) &=& \text{ordre de $q^{\deg(y)}$ dans $(\ZZ/\ell\ZZ)^\times$}.
\end{eqnarray*}
et notons $\YY_{\ell}(q,n,w)$ l'ensemble des $y\in\kk^{\times}$, 
d'ordre premier à $\ell$, 
tels que $\ee(y)$ soit égal à $w_0$, 
le plus grand diviseur de $w$ premier à $\ell$.
Nous allons définir une application surjective~:
\begin{equation*}
\YY_{\ell}(q,n,w) \to \BB_{\ell}(q,m,d,w)
\end{equation*}
dépendant du choix de $\D$. 

Soit $y\in\YY_{\ell}(q,n,w)$. 
Notons~:
\begin{equation*}
r(y) = \frac{\deg(y)}{(\deg(y),d)}
\end{equation*} 
le degré de $y$ sur $\kk_{\D}$ et $\boldsymbol{\uptau}_\D(y)$ la re\-pré\-senta\-tion 
irréductible supercus\-pidale $\ell$-modulaire du groupe $\GL_{r(y)}(\kk_\D)$ 
correspon\-dant à $y$ par \eqref{james}. 
Posons $s(y)=s(\boldsymbol{\uptau}_\D(y))$ et~:
\begin{equation*}
\d(y)=\frac{w}{(w,s(y))}.
\end{equation*}
On a le lemme suivant. 

\begin{lemm}
Il existe une unique représentation irréductible cuspidale~:
\begin{equation*}
\boldsymbol{\upsigma}_\D(y) 
\end{equation*} 
dont le support supercuspidal soit égal à $k(y)\cdot\boldsymbol{\uptau}_\D(y)$.
\end{lemm}

\begin{proof}
D'après \cite[III.2.5]{Vigb},
il suffit de prouver que 
le plus grand diviseur de $k$ pre\-mier à $\ell$, noté $k_0$, 
est égal à l'ordre de $q^{dr(y)}$ dans $(\ZZ/\ell\ZZ)^\times$.
Comme on a $\ee(y)=w_0$ d'une part et $r(y)d = s(y)\deg(y)$
d'autre part, cet ordre est égal à~:
\begin{equation*}
\frac{w_0}{(w_0,s(y))} = k_0,
\end{equation*}
ce qui met fin à la démonstration. 
\end{proof}

\begin{lemm}
Pour tout $y\in\YY_{\ell}(q,n,w)$,
la représentation $\boldsymbol{\upsigma}_\D(y)$ appartient à $\BB_{\ell}(q,m,d,w)$.
\end{lemm}

\begin{proof}
Il faut d'abord prouver que le degré de $\boldsymbol{\upsigma}_\D(y)$ divise 
$m$. 
D'abord, on a~:
\begin{equation*}
s(y)=\frac{d}{(\deg(y),d)}.
\end{equation*}
Par hypothèse sur $y$, il existe un entier $t\>1$ tel que $n=wt\cdot\deg(y)$.
On en déduit que~:
\begin{equation*}
r(y)=\frac{m}{(m,wt)}
\quad\text{et}\quad
s(y)=\frac{wt}{(m,wt)}
\quad\text{et}\quad
k(y)=\frac{(m,wt)}{((m,wt),t)},
\end{equation*}
puis que l'entier $\deg(\boldsymbol{\upsigma}_\D(y))=k(y)r(y)$ divise $m$. 
D 'après le début de la section \ref{SECTION5}, il ne reste 
qu'à vé\-ri\-fier que 
$a=(w,s(y))$ divise $s(y)$ et que $s(y)a^{-1}=s(y)(w,s(y))^{-1}$ est premier à $k(y)$, 
ce qui est im\-mé\-diat. 
\end{proof}

Ceci définit une application
$\boldsymbol{\upsigma}_\D : y \mapsto{} \boldsymbol{\upsigma}_\D(y)$
de $\YY_{\ell}(q,n,w)$ dans $\BB_\ell(q,m,d,w)$.

\begin{prop}
L'application $\boldsymbol{\upsigma}_\D$ est surjective, 
et ses fibres sont les classes de conjugaison sous l'automorphisme de 
Frobenius $x\mapsto x^{q^d}$.
\end{prop}

\begin{proof}
Soit $\s\in\BB_\ell(q,m,d,w)$. 
Il y a une unique représentation irréductible supercuspidale 
$\ell$-modulaire $\boldsymbol{\uptau}(\s)$ telle que le support supercuspidal de 
$\s$ soit égal à $\d(\s)\cdot \boldsymbol{\uptau}(\s)$. 
Soit $y\in\overline{\kk}{}^\times$ un paramètre de James pour 
$\boldsymbol{\uptau}(\s)$. 
Il est d'ordre premier à $\ell$ et vérifie $\ee(y)=w_0$, 
mais il est \textit{a priori} dans une extension de $\kk_\F$ de degré 
$n\d(\s)^{-1}$. 
Toutefois, par hypo\-thèse sur $\s$, l'entier $w\d(\s)^{-1}$ divise $s(\tau)$. 
On en déduit que $y$ est dans une extension de $\kk_\F$ de degré $nw^{-1}$, 
\ie que $y\in\YY_\ell(q,n,w)$, ce qui prouve la surjectivité. 
Ensuite, on a~:
\begin{equation}
\label{FabricedelDongo}
\deg(\s) = \frac{\deg(y)}{(\deg(y),d)}\cdot\frac{w}{(w,s(y))}
\quad\text{et}\quad
\d(\s) = (w,\deg(\s)) = \frac{w}{(w,s(y))}.
\end{equation}
Il s'ensuit que l'application $\s\mapsto\boldsymbol{\uptau}(\s)$ est injective. 
\end{proof}


On en déduit que~:
\begin{eqnarray*}
\boldsymbol{a}_{\ell}(\D,\0,w) 
&=& \sum_{y} \frac{s(\boldsymbol{\upsigma}_\D(y))}{d}\cdot\frac{1}{\deg(\boldsymbol{\upsigma}_\D(y))} \\
&=& \sum_{y} \frac{1}{\deg(y)}
\end{eqnarray*}
(où $y$ décrit $\YY_\ell(q,n,w)$),
valeur que l'on note $\boldsymbol{y}_{\ell}^{1}(q,n,w)$.

\subsection{}

Supposons maintenant que $\TT$ est quelconque, de degré
$g$ divisant $nw^{-1}$.  
On pose~:
\begin{equation*}
q(\TT)=q^{f(\TT)},
\quad
n(\TT) = m'd' = \frac{n}{\gg},
\end{equation*}
où $f(\TT)$ désigne le degré résiduel de l'endoclasse $\TT$.
Fixons un entier $u$ de la forme \eqref{quark} et une réalisation 
$([\La,n_\La,0,\b],\t)$, où 
$\t$ est un caractère simple ($\ell$-modulaire) attaché à la strate 
simple $[\La,n_\La,0,\b]$ de $\Mat_u(\D)$.
On suppose que 
l'intersection entre l'ordre hérédi\-taire associé à $\La$ 
et le cen\-tralisateur de $\F[\b]$ dans $\Mat_u(\D)$ est maximal.
Fixons aussi une $\b$-extension $\k$ de ${\t}$ et posons 
$\J=\J(\b,\La)$. 
L'application $\s\mapsto \k\otimes\s$ induit une surjection~:
\begin{equation}
\label{elevteria}
\s \mapsto [\J,\k\otimes\s] \leftrightarrow [\GL_u(\D),\rho]
\end{equation}
entre classes d'isomorphisme de représentations 
irréductibles cuspidales $\ell$-modulaires du groupe $\GL_{u'}(\kk_{\D'})$ et 
classes inertielles de représentations irréductibles cuspidales 
$\ell$-modulaires du groupe $\GL_{u}(\D)$ d'endoclasse $\TT$.
L'image d'une représentation $\s$ par \eqref{elevteria} appartient à  
$\Aa_{\ell}(\D,\TT,w)$ si et seulement si $\s$ est dans $\BB_{\ell}(q(\TT),m',d',w)$, 
et l'ensemble des antécédents d'une classe inertielle $[\GL_u(\D),\rho]$ par 
\eqref{elevteria} est de cardinal $s(\rho)$.
On obtient ainsi~:
\begin{equation*}
\boldsymbol{a}_{\ell}(\D,\TT,w) 
= \sum_{\s} \frac{s(\s)}{d'}
= \boldsymbol{y}_{\ell}^{1}(q(\TT),n(\TT),w)
\end{equation*}
(où $\s$ décrit l'ensemble $\BB_{\ell}(q(\TT),m',d',w)$). 

\subsection{}

Finalement, si l'on fixe un nombre rationnel $j\>0$,
et si l'on pose~:
\begin{equation*}
\Aa_\ell(\D,j,w) = \bigcup\limits_{l(\TT)\<j} \Aa_\ell(\D,\TT,w),
\end{equation*}
alors on obtient l'égalité~:
\begin{equation*}
\boldsymbol{a}_{\ell}(\D,j,w) 
= \sum\limits_{l(\TT)\<j} \boldsymbol{a}_{\ell}(\D,\TT,w) 
= \sum\limits_{l(\TT)\<j} \boldsymbol{y}_{\ell}^{1}(q(\TT),n(\TT),w)
= \boldsymbol{a}_{\ell}(\F,j,w),
\end{equation*}
ce qui met fin à la 
preuve du théorème \ref{comptagedeclasses} et du corollaire \ref{DonaTartt}. 



\providecommand{\bysame}{\leavevmode ---\ }
\providecommand{\og}{``}
\providecommand{\fg}{''}
\providecommand{\smfandname}{\&}


\end{document}